\documentclass{amsart}
\usepackage[english]{babel}
\usepackage{amsfonts, amsmath, amsthm, amssymb,amscd,indentfirst}

\usepackage{esint}
\usepackage{amsmath,amssymb,latexsym,indentfirst}
\usepackage{times}
\usepackage{palatino}

\newtheorem{theorem}{Theorem}[section]

\newtheorem{proposition}{Proposition}[section]
\newtheorem{lemma}{Lemma}[section]
\newtheorem{definition}{Definition}[section]

\newtheorem{remark}{Remark}[section]

\def\Sc{{\mathcal S}}

\def\Yc{{\mathcal Y}}

\begin{document}

\title[Charged black holes]{Penrose inequalities and a positive mass theorem for charged black holes in higher dimension}

\author{Levi Lopes de Lima}
\author{Frederico Gir\~ao}
\author{Weslley Loz\'orio}
\address{Federal University of Cear\'a,
Department of Mathematics, Campus do Pici, Av. Humberto Monte, s/n, Bloco 914, Pici, 60455-760,
Fortaleza/CE, Brazil.}
\email{levi@mat.ufc.br}
\email{fred@mat.ufc.br}
\email{w2weslley@yahoo.com.br}
\author{Juscelino Silva}
\address{Federal University of Piau\'{\i}, Department of Mathematics, Campus Petronio Portela, Ininga, 64049-550, Teresina/PI, Brazil}
\email{jsilva@ufpi.edu.br}
\thanks{The first and second authors were partially supported by CNPq/Brazil grants. The first and last authors were partially supported by a CAPES/Brazil grant.}

\subjclass[2010]{Primary: 53C21, Secondary: 53C80, 53C44, 53C42}

\begin{abstract}
We use the inverse mean curvature flow to establish Penrose-type inequalities for time-symmetric Einstein-Maxwell initial data sets which can be suitably embedded as a hypersurface in Euclidean space $\mathbb R^{n+1}$, $n\geq 3$. In particular, we prove a positive mass theorem for this class of charged black holes.
As an application we show that the conjectured  upper bound for the area in terms of the mass and the charge, which in dimension $n=3$ is relevant in connection with the Cosmic Censorship Conjecture, always holds under the natural assumption that the horizon is stable as a minimal hypersurface.
\end{abstract}

\maketitle

\section{Introduction}
\label{intro}

Due mainly to its aesthetical appeal and increasingly experimental confirmation, General Relativity is usually regarded as a highly successful theory of space, time and matter at the classical level. It is well-known, however, that the theory still lacks physical consistency in respect to  the long-time behavior of solutions of the corresponding initial value formulation. More precisely, since it has been shown by Penrose and Hawking \cite{HE} that, under fairly natural assumptions, an initial data set propagates to a geodesically incomplete space-time, physical consistency  can only be achieved under the assumption that these singularities can not be seen by remote observers, being shielded from them by an event horizon. Roughly speaking, this is the content of the famous Cosmic Censorship Conjecture (CCC).

The notorious difficulty in deciding on the validity of this conjecture, which remains a fundamental open problem in gravitational collapse phys\-ics, is of course related to the fact that it concerns the long-term  behavior of solutions. This led Penrose \cite{P} to devise a test for falsifying the conjecture, which happens to be  formulated solely in terms of the initial data set. More precisely, he argued that if the conventional picture of gravitational collapse (CCC included) is taken for granted then there necessarily holds a lower bound for the total mass of vacuum black hole solutions in terms of the area of the outermost minimal surface enclosing the apparent horizon. Moreover, this bound should be saturated only by the Schwarzschild solution, which describes a spherically symmetric vacuum black hole.
In the asymptotically flat case, which corresponds to a vanishing cosmological constant, this inequality has been established for time-symmetric initial data sets by Huisken-Ilmanen \cite{HI} and Bray \cite{B} in dimension $n=3$ and by Bray-Lee \cite{BL} for $n\leq 7$; see
\cite{Ma} for a rather complete survey on these developments.
Recent contributions to the graph and conformally flat cases in all dimensions appear in \cite{L} \cite{dLG1} \cite{dLG2} \cite{FS} \cite{HW} \cite{MV}.

We mention that recently there has been much interest in extending this type of result to initial data sets arising from vacuum solutions of Einstein's field equations with a negative cosmological constant.
For instance, in dimension $n=3$ Lee and Neves \cite{LN} were able to use the inverse mean curvature flow to establish a Penrose-type inequality for  asymptotically locally hyperbolic manifolds in the so-called negative mass range. By using the method developed in \cite{dLG3} and \cite{dLG4},
where the conjectured Penrose inequalities have been established when the manifold can be appropriately embedded as a graph in hyperbolic space (see also \cite{DGS} for a previous contribution to the subject) the first two authors proved
an optimal Penrose inequality for 
a certain class of asymptotically locally hyperbolic graphs in any dimension $n\geq 3$ \cite{dLG5}. A key ingredient in this approach is a sharp Alexandrov-Fenchel-type inequality for a class of hypersurfaces in certain locally hyperbolic manifolds. This geometric inequality clearly has an independent interest and its proof uses the inverse mean curvature flow. 
In this regard we note that lately there has been considerable progress in establishing Alexandrov-Fenchel-type inequalities in a large class of warped products, including space forms; see for instance \cite{BHW} 
\cite{GWWX}
\cite{WX} and the references therein.

These advances settle, for initial data sets which can be suitably embedded as
graphs in appropriate ambient spaces, the conjectured Penrose inequalities for vacuum black holes, where only the gravitational contribution to the total energy is taken into account. The purpose of this note is to show how the graph method can be adapted to yield  Penrose-type inequalities for {\em charged} black holes in any dimension. Our  contributions in this direction (Theorems \ref{main0}, \ref{main1}, \ref{main2} and \ref{main3} below) do not provide the sharp form of the conjectured inequality in general but they show at least that the charge actually contributes to the total mass of the system in a way that depends explicitly on integral-geometric invariants of the horizon, with no reference whatsoever to the initial data set as a whole. In particular, we prove the optimal version of the positive mass theorem for this class of charged black holes, a result that in general does not seem to be approachable by other means in dimension $n\geq 4$.

Classically, i.e. in dimension $n=3$, the Penrose inequality for time-sym\-metric charged black holes with a connected horizon follows from the standard Geroch-Jang-Wald-Huis\-ken-Ilmanen monotonicity method based on the weak formulation of the inverse mean curvature flow. A detailed proof can be found in \cite{DK}, where actually the result is extended to the non-time-symmetric case under the assumption that a certain coupled Jang-IMCF system admits a solution.
Similarly to what happens in dimension $n=3$, the conjectured Penrose inequality (\ref{penineconj}) is equivalent to an upper and lower bound for the area of the horizon in terms of certain expressions involving the mass and the charge; see (\ref{penineconj2}) below.
In the classical situation, connectedness of the horizon is essential for the validity of the full inequality in view of the well-known counterexample in \cite{WY}. In fact, this counterexample only violates the lower bound, which is consistent with the fact that, as already   observed by Jang \cite{J}, in general only the upper bound follows from Penrose's argument.
We note, however, that in \cite{KWY1} it has been observed that, for connected horizons, 
the lower bound can be justified on physical grounds as well in the sense that it can be reduced 
to the positive mass theorem with charge, which of course is based on
CCC through Penrose's heuristic arguments; see also \cite{DKWY} where  the most general form of these lower
bounds has been established in a similar fashion.
On the other hand,  
the upper bound
has been recently established for charged black holes with multiply connected horizons in \cite{KWY2}, by means of Bray's conformal flow.
In our higher dimensional setting, it turns out that the positive mass inequality mentioned above immediately gives a  proof of the  upper bound for {\em any} graphical time-symmetric initial data set with a connected horizon which is stable as a minimal hypersurface; see Theorem \ref{main3} below. 
Also, we mention that
a Penrose-like inequality for general (not necessarily time-symmetric) initial data sets can be found in \cite{K}. We also note that an array of similar inequalities has been recently established  for dynamical black holes, notably in the axially symmetric case; we refer to \cite{Ma} and \cite{D} for  surveys of these aspects of the theory. Finally, we observe that most of the results presented here admit counterparts in the negative cosmological case. This issue will be addressed in a companion paper \cite{dLGLS}.

This paper is organized as follows.
In Section \ref{rnsol} we recall the  so-called Reissner-Nordstr\"om-Tangherlini (RNT) solution, which describes a spherically symmetric charged black hole in any space-time dimension $n+1\geq 4$.
The precise relationship between mass, area of the horizon and charge for this solution motivates the formulation of a Penrose inequality for time-symmetric charged black holes, which extends to higher dimensions the classical conjecture put forward by Jang \cite{J} for $n=3$; this is proposed in Section \ref{propo}. Finally, in Section \ref{main} we present the proofs of our main results.
The key ingredient is Proposition \ref{lowerb}, whose proof is based on the inverse mean curvature flow.

\vspace{0.2cm}
\noindent
{\bf Acknowledgements:} The authors would like to thank  S. Dain for pointing out that the graph method could provide effective results for charged black holes. Also, the authors would like to thank  M. Disconzi, M. Khuri and F. Schwartz for valuable comments and M. Mars for reading a previous version of this manuscript and contributing with suggestions that helped to substantially improve the presentation.

\section{The Reissner-Nordstr\"om-Tangherlini solution}\label{rnsol}

We start by recalling the so-called RNT solution in space-time dimension
$n+1\geq 4$,
given by
\begin{equation}\label{staticans}
\overline g_{m,q}=-\psi^2dt^2+g_{m,q}, \quad g_{m,q}=\psi^{-2}dr^2+r^2 h,
\end{equation}
where $t\in\mathbb R$, $r$ is a radial parameter varying in a certain open interval $I\subset (0,+\infty)$ and
\begin{equation}\label{psiform}
\psi(r)=\psi_{m,q}(r)=\sqrt{1-\frac{2m}{r^{n-2}}+\frac{q^2}{r^{2n-4}}},
\end{equation}
where $m$ and $q$ are real parameters \cite{T} \cite{KI}.
Thus, $\overline g=\overline g_{m,q}$ is defined on the product manifold $\mathbb R\times P$, $P=I\times \mathbb S^{n-1}$, where $(\mathbb S^{n-1},h)$ is  the standard round sphere.
These metrics are asymptotically flat, spherically symmetric solutions of the Einstein-Maxwell equations
\begin{equation}\label{field}
{\rm Ric}_{\overline g}-\frac{R_{\overline g}}{2}\overline g =T_F,\quad dF=0, \quad {\rm div}_{\overline g}F=0.
\end{equation}
where
${\rm Ric}_{\overline g}$ is the Ricci tensor of $\overline g$ and $R_{\overline g}={\rm tr}_{\overline g}{\rm Ric}_{\overline g}$ is the scalar curvature.
In this setting, the metric $\overline g$ represents the gravitational potential and $T_F$ is the energy-momentum tensor associated to the electromagnetic skew-symmetric $2$-form $F$ in the standard manner: we take
$T_F$ proportional to
\begin{equation}\label{enmomt}
\tilde T_{ij}=F_{ik}F_j^{\,k}-\frac{1}{4}F_{km}F^{km}\overline g_{ij}.
\end{equation}
For simplicity we assume that no  magnetic fields are present.

As we shall see in next section, the parameter $m>0$ can be identified with the total mass of the solution.
On the other hand,
the parameter $q\in \mathbb R$ is the electric charge, which is given by
\begin{equation}\label{chargeflux}
q=\frac{1}{\omega_{n-1}}\int_{\mathbb S_{r}}\langle E,\nu_{r}\rangle d\mathbb S_{r},
\end{equation}
where $\omega_{n-1}={\rm area}_{n-1}(\mathbb S^{n-1},h)$, $\mathbb S_{r}=\{r\}\times \mathbb S^{n-1}\subset P$ is a spherical slice, $\nu_{r}=\psi\partial_r$ is its outward unit normal and $E=q\nu_r/r^{n-1}$ is the electric field. This is consistent with the easily checked fact that ${\rm div}_{g_{m,q}}E=0$.

The nature of the interval $I$ is clarified by looking at the zeros of the expression inside the square root sign on the left-hand side of (\ref{psiform}), which are given by
\[
r_{\pm}=\left(m\pm\sqrt{m^2-q^2}\right)^{\frac{1}{n-2}}.
\]
Here we  assume that
\begin{equation}\label{pmtrnt}
m>|q|,
\end{equation}
so that only the outermost zero $r_+$ is relevant to our purposes. Thus, we take $I=(r_+,+\infty)$. In this case it can be shown that the solution (\ref{staticans})-(\ref{psiform}) extends smoothly to the interval $\overline I=[r_+,+\infty)$, with the hypersurface $r=r_+$ being null. Therefore, the RNT solution  has the causal structure of a charged black hole indeed.

\begin{remark}
\label{extremernt}
{\rm In case $m=|q|$ the horizons $r=r_\pm$ merge to yield a single horizon for the so-called {\em extremal RNT solution}, which plays a central role in black hole thermodynamics in the context of String Theory \cite{KL}. On the other hand, the case $m<|q|$ is not physically interesting since it displays a naked singularity at $r=0$.} 
\end{remark}

Regarding the Riemannian metric $g_{m,q}$ on the initial data set $t=0$,
notice that it satisfies
\begin{equation}\label{asymptflat}
|g_{m,q}-g_{0}|_{g_{0}}+ |dg_{m,q}|_{g_{0}}=O\left(r^{-(n-2)}\right),
\end{equation}
where $g_0=g_{0,0}=dr^2+r^2h$ is the standard Euclidean metric, thus confirming its asymptotically flat character; see Definition \ref{asymfdef} below.
Moreover,
its scalar curvature is
\[
R_{g_{m,q}} =  (n-1)(n-2)\frac{q^2}{r^{2n-2}},
\]
or in terms of the electric field,
\begin{equation}\label{neue}
R_{g_{m,q}} 
  =   (n-1)(n-2)|E|^2_{g_{m,q}}
 \end{equation}
Finally, note that the mass parameter $m$ can be expressed in terms of the charge $q$ and the area $|\mathbb S_{r_+}|$ of the horizon as
\begin{equation}\label{massf}
m=\frac{1}{2}\left(\left(\frac{|\mathbb S_{r_+}|}{\omega_{n-1}}\right)^{\frac{n-2}{n-1}}+
q^2\left(\frac{\omega_{n-1}}{|\mathbb S_{r_+}|}\right)^{\frac{n-2}{n-1}}\right).
\end{equation}

\section{Penrose inequalities for charged black holes}
\label{propo}

The asymptotic expansion (\ref{asymptflat}) suggests the consideration of the following well-known class of Riemannian manifolds.

\begin{definition}\label{asymfdef}
Let $(M^n,g)$ be a complete $n$-dimensional Riemannian manifold, possibly carrying a compact inner boundary $\Sigma$. For simplicity, we assume that $M$ has a unique end, say $E$. We say that $(M,g)$ is {\em asymptotically flat} (AF) if there exists a chart $\Psi$ taking $E$ to $\mathbb R^n-B_1(0)$, the complement of a ball in Euclidean space, so that, as $r\to +\infty$,
\begin{equation}\label{alhcond}
|\Psi_*g-g_{0}|_{g_{0}}+ |d\Psi_*g|_{g_{0}}=O\left(r^{-\sigma}\right),
\end{equation}
for some $\sigma>(n-2)/2$. We also assume that $R_{\Psi_*g}$ is integrable, where, as usual, $R$ denotes scalar curvature.
\end{definition}

For this kind of manifold,  the so-called {\em ADM mass} is defined as
\begin{equation}\label{massl}
\mathfrak m_{(M,g)}=\lim_{r\to +\infty}c_n\int_{\mathbb S^{n-1}_r}
  ({\rm div}_{g_{0}}e-d{\rm tr}_{g_{0}}e)(\nu_r)d\mathbb S_r,
\end{equation}
where $e=\Psi_*g-g_{0}$, $\mathbb S^{n-1}_r\subset\mathbb R^n-B_1(0)$  is a sphere of radius $r$, $\nu_r$ is its {\em outward} unit vector  and
\[
c_n=\frac{1}{2(n-1)\omega_{n-1}},
\]
where $\omega_{n-1}={\rm area}_{n-1}(\mathbb S^{n-1})$.

Let us consider from now on a time-symmetric Einstein-Maxwell initial data set $(M,g,E)$.
This means that
$(M,g)$ is an AF manifold of dimension $n\geq 3$ carrying an outermost compact and minimal hypersurface $\Sigma_0$ (the horizon) and satisfying the corresponding Hamiltonian constraint equation
\begin{equation}\label{const}
2\mu=R_g-(n-1)(n-2)|E|^2_g,
\end{equation}
where $\mu$ is the energy density and
 $E$ (the electric field) is a tangent vector field whose charge density vanishes:
\begin{equation}\label{chargde}
{\rm div}_gE=0.
\end{equation}
Moreover, in the given asymptotically flat coordinates at infinity, there holds
\begin{equation}\label{asymchar}
|E|=O(r^{-(n-1)}),
\end{equation}
so that the charge is defined by
\begin{equation}\label{chargedef}
Q=\lim_{r\to +\infty}\frac{1}{\omega_{n-1}}\int_{\mathbb S_r}\langle E,\nu_r\rangle d\mathbb S_r.
\end{equation}
In fact, in view of (\ref{chargde}) it follows that $Q$  is a quasi-local quantity in the sense that
\begin{equation}
Q=\frac{1}{\omega_{n-1}}\int_{\Sc}\langle E,\nu\rangle d\mathbb \Sc,
\end{equation}
for any hypersurface $\Sc\subset P$ which is homologous to the sphere at infinity.
Finally, we assume that the
corresponding dominant energy condition  $\mu\geq 0$ holds, which means that
\begin{equation}\label{decond}
R_g\geq (n-1)(n-2)|E|^2_g.
\end{equation}
Note that this energy condition differs from the standard one, since the contribution coming from the electromagnetic field has been discarded.

The discussion in the previous section shows that
the RNT solution $(P,g_{m,q})$ meets all of these conditions. Moreover, a straightforward computation shows that $\mathfrak m_{(Pg_{m,q})}=m$. Thus, in view of (\ref{massf}), it is natural to conjecture that for any initial data $(M,g)$ as above, the following Penrose inequality should hold:
\begin{equation}\label{penineconj}
{\mathfrak m}_{(M,g)}\geq \frac{1}{2}\left(\mathfrak R_{\Sigma_0}+
\frac{Q^2}{\mathfrak R_{\Sigma_0}}\right),
\quad \mathfrak R_{\Sigma_0}=\left(\frac{|\Sigma_0|}{\omega_{n-1}}\right)^{\frac{n-2}{n-1}},
\end{equation}
with the equality occurring if and only if $(M,g)$ is isometric to $(P,g_{m,q})$.
At least in dimension $n=3$, this conjecture is extensively discussed in  the literature; see for instance the classical presentation \cite{J}
and \cite{Ma} \cite{DK} \cite{K} \cite{KWY1} \cite{KWY2}
\cite{WY} for more modern accounts.
In this classical setting, we note that connectedness of $\Sigma_0$ is required in view of the counterexample in \cite{WY}, which is based on the so-called Majumdar-Papapetrou metrics and very likely exists in any dimension. 

\begin{remark}
\label{posrem}
{\rm Notice that, in the presence of the positive mass inequality
\begin{equation}
\label{posimassass}
\mathfrak m_{(M,g)}\geq |Q|,
\end{equation} 
which is conjectured to hold for any initial data set as above and is proved here in the graphical context (Theorem \ref{main0} below), 
(\ref{penineconj}) is equivalent to the inequalities
\begin{equation}\label{penineconj2}
{\mathfrak m}_{(M,g)}-\sqrt{{\mathfrak m}_{(M,g)}^2-Q^2}\leq\mathfrak R_{\Sigma_0}\leq
{\mathfrak m}_{(M,g)}+\sqrt{{\mathfrak m}_{(M,g)}^2-Q^2}.
\end{equation}
We note that  
the counterexample in \cite{WY} only violates the lower bound in (\ref{penineconj2}).
}
\end{remark}

In this paper we discuss the validity of these inequalities for initial data sets which can be appropriately embedded as hypersurfaces in Euclidean space $\mathbb R^{n+1}$. In our setting, the next definition finds its motivation in the fact that the RNT solution $g_{m,q}$  can be  isometrically embedded in $\mathbb R^{n+1}$ as the graph of a radially symmetric function $u=u(r)$, $r>r_+$, satisfying 
\begin{equation}
\label{realrnt}
\left(\frac{du}{dr}\right)^2=\frac{2mr^{n-2}-q^2}
{r^{2n-4}-2mr^{n-2}+q^2}.
\end{equation}

\begin{definition}\label{alshyp}
We say that a complete hypersurface $(M,g)\subset \mathbb R^{n+1}$, possibly carrying a compact inner boundary $\Sigma_0$, is {\em asymptotically flat} if there exists a compact set $K\subset M$ so that $M-K$ can be written as a graph over the end $E_0$ of the horizontal slice
 $\mathbb R^n_t= \{t\}\times \mathbb R^n$, with the graph being  associated to a smooth function $u$ such
that, in the nonparametric chart $\Psi_u(x,u(x))=x$, $x\in E_0$, there holds
\[
u_i(x)=O(|x|^{-\sigma/2}), \quad |x|u_{ij}(x)+|x|^2u_{ijk}
=O(|x|^{-\sigma/2}),
\]
as $|x|\to +\infty$
for some $\sigma>(n-2)/2$.
Moreover, we assume that $R_{{\Psi_u}_*g}$ is integrable.
\end{definition}

Under these conditions, the mass $\mathfrak m_{(M,g)}$ can be computed by taking $\Psi=\Psi_u$ in (\ref{massl}). More precisely, if we assume further that the inner boundary $\Sigma_0$ lies on some totally geodesic horizontal slice $\mathbb R^n_t$, which we of course identify with $\mathbb R^n$, and moreover that the intersection  $M\cap \mathbb R^n_t$ is orthogonal along $\Sigma$, so that $\Sigma\subset M$ is totally geodesic (and hence a horizon indeed) then the computations in \cite{dLG2} provide the following integral formula for the mass:
\begin{equation}\label{massform}
\mathfrak m_{(M,g)}=c_n\int_{\Sigma_0}  Hd\Sigma +c_n\int_M\Theta R_g dM,
\end{equation}
where $\Theta=\langle \partial/\partial t,N\rangle$, $N$ is the unit normal to $M$, which we choose so as to point upward at infinity, and $H$ is the mean curvature of $\Sigma\subset \mathbb R^n_t$ with respect to its {inward} pointing unit normal.

This expression is the starting point in the proof of our main results, which we now pass to describe.
The first one provides a lower bound for the mass in terms of the total mean curvature of the horizon viewed as a hypersurface in $\mathbb R^n$. Recall that a hypersurface $(\Sigma_0,k)\subset (\mathbb R^{n},g_0)$ is {\em mean-convex} if it satisfies $H\geq 0$ and $2$-{\em convex} if it further satisfies $R_k\geq 0$.

\begin{theorem}\label{main0}
Let $(M,g)$ be an Einstein-Maxwell initial data set which can be isometrically embedded in $\mathbb R^{n+1}$ as explained above. Assume also that its horizon $(\Sigma_0,k)$, viewed as a hypersurface in $(\mathbb R^{n},g_0)$, is mean convex and  star-shaped. Then there holds
\begin{equation}\label{mainineq}
\mathfrak m_{(M,g)}\geq \frac{1}{2}\left(\left(2c_n\int_{\Sigma_0}Hd\Sigma_0\right)+ Q^2\left(2c_n \int_{\Sigma_0}Hd\Sigma_0\right)^{-1}\right),
\end{equation}
In particular, the positive mass theorem holds for $(M,g)$:
\begin{equation}\label{pmtem}
\mathfrak m_{(M,g)}\geq |Q|.
\end{equation}
Moreover, if the equality holds in (\ref{pmtem}) then we actually have
\begin{equation}\label{pmtem2}
\mathfrak m_{(M,g)}\geq \mathfrak R_{\Sigma_0},
\end{equation}
so that $(M,g)$ is {\em not} isometric to a RNT solution with $m>|q|$.
\end{theorem}

Notice that (\ref{pmtem}) follows from (\ref{mainineq}) and the arithmetic-geometric inequality. Also, if the equality holds in (\ref{pmtem}) then we necessarily have
\[
|Q|=2c_n\int_{\Sigma_0}Hd\Sigma_0.
\]
On the other hand, since $\Sigma_0\subset\mathbb R^n$ is mean convex and star-shaped, the main result in \cite{GL} gives
\begin{equation}\label{afeuc}
2c_n\int_\Sigma  Hd\Sigma\geq \mathfrak R_{\Sigma_0}.
\end{equation}
Combining  this with (\ref{mainineq}), (\ref{pmtem2}) follows.

\begin{remark}\label{extremerem}
{\em We observe that (\ref{pmtem}) gives the optimal form of the conjectured positive mass inequality for charged black holes; see Remark \ref{posrem}. Moreover, it is consistent  with (\ref{pmtrnt}), which appears as a natural assumption in order to make sure that the RNT solution has the causal structure of a black hole. To the authors' knowledge, this result has only been established in dimension $n=3$, either  using  spinors \cite{GHHP} or the inverse mean curvature flow \cite{DK}. 
}
\end{remark}

\begin{remark}
\label{rigidrem}
{\rm The rigidity statement in Theorem \ref{main0} provides some  indication that equality in (\ref{pmtem}) should imply that $(M,g)$ is isometric to an extreme RNT solution; see  \cite{KW} and Remarks \ref{extremernt} and \ref{rigista}. }  
\end{remark}

To explain our next result, recall that if $(U,k)$ is a  (not necessarily connected) closed Riemannian manifold of dimension $n-1\geq 2$, then its {\em Yamabe quotient} is 
\begin{equation}\label{yamquot}
\Yc_{(U,k)}=\frac{\int_U R_kdU}{\left(\int_U dU\right)^{\frac{n-3}{n-1}}}.
\end{equation}
Notice that $\Yc_{(U,k)}$ is scale invariant.
Moreover, we define the {\em relative Yamabe quotient} of $(U,k)$ as
\begin{equation}\label{yamquotrel}
\widetilde \Yc_{(U,k)}=\frac{\Yc_{(U,k)}}{\Yc_{(\mathbb S^{n-1},g_0)}},
\end{equation}
where $g_0$ is the standard round metric in $\mathbb S^{n-1}$.

\begin{theorem}\label{main1}
Let $(M,g)$ be an Einstein-Maxwell initial data set which can be isometrically embedded in $\mathbb R^{n+1}$ as above. Assume also that its horizon $(\Sigma_0,k)$ is mean convex, star-shaped, stable as a minimal hypersurface in $(M,g)$ and satisfies
\begin{equation}\label{condmain2}
\widetilde\Yc_{(\Sigma_0,k)}\leq 1.
\end{equation}
Then the optimal Penrose inequality (\ref{penineconj}) holds for $(M,g)$.
\end{theorem}

The assumption that $\Sigma_0$ is star-shaped implies, by Gauss-Bonnet, that $\tilde \Yc_{(\Sigma_0,k)}=1$ if $n=3$. Thus, our method recovers the optimal Penrose inequality (\ref{penineconj}) in this dimension.
If $n\geq 4$ then
$\tilde\Yc_{(\Sigma_0,k)}$ is always positive under our assumptions. Actually, the explicit lower bound
\begin{equation}\label{rew}
\widetilde\Yc_{(\Sigma_0,k)}\geq Q^2\mathfrak R_{\Sigma_0}^{-2}
\end{equation}
follows from the proof of Theorem \ref{main1}; see (\ref{rew2}) below. We note however that it might well happen that  $\tilde\Yc_{(\Sigma_0,k)}>1$. Indeed, the Alexandrov-Fenchel inequalities proved in \cite{GL} imply  that $\tilde\Yc_{(\Sigma_0,k)}\geq1$ if we assume that $\Sigma_0$ is $2$-convex and star-shaped, with the equality holding if and only if $\Sigma_0$ is a round sphere. This shows that Theorem \ref{main1} does not apply in case $\Sigma_0$ is $2$-convex and {\em not} congruent to a round sphere.
This case is covered by the following result, which in particular provides a Penrose-type inequality for convex horizons.

\begin{theorem}\label{main2}
Let $(M,g)$ be an Einstein-Maxwell initial data set which can be isometrically embedded in $\mathbb R^{n+1}$ as above. Assume that $n\geq 4$ and the horizon $(\Sigma_0,k)$ is $2$-convex and star-shaped. Then the following Penrose-type inequality holds for $(M,g)$:
\begin{equation}\label{resmain1}
{\mathfrak m}_{(M,g)}\geq \frac{1}{2}\left(\mathfrak R_{\Sigma_0}+
\widetilde\Yc_{(\Sigma_0,k)}^{-\frac{n-2}{n-3}}\frac{Q^2}{\mathfrak R_{\Sigma_0}}\right).
\end{equation}
\end{theorem}

We remark that the assumption (\ref{condmain2}) is only used in the proof of Theorem \ref{main1} in order to make sure via (\ref{rew}) that $|Q|\leq\mathfrak R_{\Sigma_0}$ holds.
If instead we have $\mathfrak R_{\Sigma_0}\leq |Q|$ then the positive mass inequality in (\ref{pmtem}) immediately leads to the following result.

\begin{theorem}\label{main3}
Let $(M,g)$ be an Einstein-Maxwell initial data set which can be isometrically embedded in $\mathbb R^{n+1}$ as  above. Assume also that its horizon $(\Sigma_0,k)$ is mean convex, star-shaped and stable as a minimal hypersurface in $(M,g)$. Then the upper bound in (\ref{penineconj2}) holds for $(M,g)$.
\end{theorem}

\begin{remark}
\label{msy} {\rm In dimension $n=3$ the stability hypothesis appearing in Theorems \ref{main1} and \ref{main3}
above follows from the much more natural assumption that the horizon $\Sigma_0\subset M$ is outer minimizing \cite{WY}. }
\end{remark}

\begin{remark}\label{outer}
{\rm  
Theorem \ref{main3} shows that the upper bound in (\ref{penineconj2}) always holds for the class of
initial data sets we are considering. Notice that the star-shapedness assumption implies that $\Sigma_0$ is necessarily connected. Nevertheless, similarly to what happens in dimension $n=3$ \cite{KWY2}, Theorem \ref{main3} can be extended to a large class of graphical initial data sets with multiply connected horizons in any dimension $n\geq 4$. The key point here is that, as explained in Remark \ref{schwartz} below, Theorem \ref{main0} still holds if, instead of being star-shaped, we assume that $\Sigma_0\subset \mathbb R^n$ is {\em outer minimizing} in the sense that any hypersurface $\Sigma'\supset \Sigma_0$ satisfies $|\Sigma'|\geq |\Sigma_0|$. Such a hypersurface is necessarily mean convex and a typical example is a finite collection of mean convex hypersurfaces which are far apart from each other. We note however that for $n=3$ this improvement fails to provide examples of multiply connected horizons for which the full Penrose inequality (\ref{penineconj}) holds as in Theorem \ref{main1}. In effect, if $\Sigma\subset\mathbb R^3$ is a collection of $p$ topological spheres then Gauss-Bonnet gives $\widetilde\Yc_{(\Sigma_0,k)}=p$, which is incompatible with (\ref{condmain2}) if $p\geq 2$. This is of course consistent with the existence of the counterexamples in \cite{WY}.}
\end{remark}

\section{The proofs of the main results}\label{main}

We now present the proofs of Theorems \ref{main0}, \ref{main1}, \ref{main2} and \ref{main3}. If $M$ is a graph associated to a function $u:\Omega\to \mathbb R$ defined on the exterior $\Omega$ of $\Sigma_0$ then
$\Theta=W^{-1}$, where $dM=Wd\Omega=\sqrt{1+|\nabla_{g_0}u|^2}d\Omega$ is the volume element of the graph in Euclidean coordinates.
Combining this with (\ref{decond}) we thus obtain from (\ref{massform}) that
\begin{equation}\label{massform2}
\mathfrak m_{(M,g)}\geq c_n\int_{\Sigma}Hd\Sigma+
(n-1)(n-2)c_n\int_\Omega |E|^2_g d\Omega.
\end{equation}

\begin{lemma}\label{dkhu}
There exists  a vector field $\hat E$ on $\Omega$ satisfying $|E|_g\geq |\hat E|_{g_0}$ and ${\rm div}_{g_0}\hat E=0$. Moreover, the charge $\hat Q$ defined by means of $\hat E$ using the prescription (\ref{chargedef}) above satisfies $\hat Q=Q$.
\end{lemma}

\begin{proof}
In the product $M\times\mathbb R$ with the Lorentzian metric $\hat g=g-dt^2$ consider the vertical graph $\hat \Omega$ defined by the function $u$. The induced metric is
\[
\hat g_{ij}=g_{ij}-u_iu_j=\delta_{ij}+u_iu_j-u_iu_j=\delta_{ij},
\]
that is, $\hat \Omega$ is an isometric copy of $\Omega$. We may apply \cite[Appendix A]{DK} to obtain a vector field $\hat E$ on $\hat\Omega$ with the desired properties. Transplanting this to $\Omega$ yields the result.
\end{proof}

The following proposition is the key result in the paper. It provides a lower bound for the bulk integral in the right-hand side of (\ref{massform2}) in terms of the total mean curvature of the horizon viewed as a hypersurface in $\mathbb R^n$ and completes the proof of Theorem \ref{main0}.

\begin{proposition}\label{lowerb}
Under the conditions of Theorem \ref{main0}, there holds
\begin{equation}\label{lowerb2}
(n-1)(n-2)c_n\int_\Omega |E|^2_g d\Omega\geq \frac{Q^2}{2}\left(2c_n
\int_{\Sigma_0}Hd\Sigma_0\right)^{-1}.
\end{equation}
\end{proposition}

\begin{proof}
From Lemma \ref{dkhu} we have
\[
(n-1)(n-2)c_n\int_\Omega |E|^2_g d\Omega\geq \frac{n-2}{2\omega_{n-1}}\int_\Omega |\hat E|^2_{g_0} d\Omega.
\]
Without loss of generality we may assume that $\Sigma_0$ is {\em strictly} mean convex ($H>0$) and star-shaped. It follows from \cite{G} \cite{U} that $\Omega$ is foliated by the leaves $\{\Sigma_t\}_{t>0}$ of the unique solution to the inverse mean curvature flow having $\Sigma_0$ as initial hypersurface.
Also, the prescription $\Sigma_t=\phi^{-1}(t)$ defines a  smooth function $\phi:\Omega\to[0,+\infty)$ satisfying $|\nabla_{g_0}\phi|_{g_0}=H$, where $H$ denotes the mean curvature of the leaves. Thus, we may  apply the co-area formula and Cauchy-Schwarz to obtain
\begin{eqnarray*}
\int_\Omega |\hat E|^2_{g_0} dx & = & \int_0^{+\infty}\left(\int_{\Sigma_t}
\frac{|\hat E|^2_{g_0}}{H}d\Sigma_t\right)dt\\
& \geq & \int_0^{+\infty}\left(\int_{\Sigma_t}
\frac{\langle \hat E,\nu_t\rangle^2_{g_0}}{H}d\Sigma_t\right)dt\\
& \geq & \int_0^{+\infty}\left(\left(\int_{\Sigma_t}Hd\Sigma_t\right)^{-1}\left(\int_{\Sigma_t}
{\langle \hat E,\nu_t\rangle_{g_0}}d\Sigma_t\right)^2\right)dt.
\end{eqnarray*}
Here, $\nu_t$ is of course the unit normal to $\Sigma_t$.
By Lemma \ref{dkhu},
$\int_{\Sigma_t}{\langle \hat E,\nu_t\rangle_{g_0}}d\Sigma_t=\omega_{n-1}Q$,  so we get
\begin{equation}\label{inter}
(n-1)(n-2)c_n\int_\Omega |E|^2_g dx\geq \frac{1}{2}(n-2)\omega_{n-1}Q^2\int_0^{+\infty}
\left(\int_{\Sigma_t}Hd\Sigma_t\right)^{-1}dt.
\end{equation}
We now observe that along the inverse mean curvature flow
there holds
\begin{eqnarray*}
\frac{d}{dt}\int_{\Sigma_t}Hd\Sigma_t  & = & \int_{\Sigma_t}\frac{H^2-|A|^2}{H}d\Sigma_t\\
& = & \int_{\Sigma_t}\frac{2K}{H}d\Sigma_t\\
& \leq & \frac{n-2}{n-1}\int_{\Sigma_t}Hd\Sigma_t,
\end{eqnarray*}
where $A$ is the shape operator of $\Sigma_t$, $K$ is the extrinsic scalar curvature (the sum of products of principal curvatures) and in the last step we  used Newton-Maclaurin's inequality.
Upon integration we get
\begin{equation}\label{chain}
\left(\int_{\Sigma_t}Hd\Sigma_t\right)^{-1}\geq \left(\int_{\Sigma_0}Hd\Sigma_0\right)^{-1}e^{-\frac{n-2}{n-1}t},
\end{equation}
and combining this with (\ref{inter}) we obtain
\begin{eqnarray*}
(n-1)(n-2)c_n\int_\Omega |E|^2_g d\Omega & \geq &\frac{1}{2}(n-2)\omega_{n-1}Q^2
\left(\int_{\Sigma_0}Hd\Sigma_0\right)^{-1}\int_0^{+\infty}e^{-\frac{n-2}{n-1}t}dt\\
& = & \frac{Q^2}{4c_n}\left(\int_{\Sigma_0}Hd\Sigma_0\right)^{-1},
\end{eqnarray*}
which proves (\ref{lowerb2}).
\end{proof}

\begin{remark}\label{rigista}
{\rm If the equality holds in (\ref{lowerb2}) then each $\Sigma_t$ is a round sphere. In particular, the horizon $\Sigma_0$ is spherical. Moreover, the scalar curvature of the graph is constant along each $\Sigma_t$, being given by
\[
R_g=(n-1)(n-2)|E|^2_g.
\]
Notice that, by (\ref{neue}), the scalar curvature of the RNT solution satisfies a similar equation. This is an indication  that equality in (\ref{mainineq}) should imply that the $(M,g)$ is congruent to the graph realization of some RNT solution. We hope to address this question elsewhere.
}
\end{remark}

Our next result plays a key role in the proof of Theorem \ref{main1}.

\begin{proposition}\label{usef}
Let $(M,g)$ be an Einstein-Maxwell initial data set as in Theorem \ref{main1}. Assume moreover that it satisfies
\begin{equation}\label{auxest}
|Q|\leq \mathfrak R_{\Sigma_0}.
\end{equation}
Then the Penrose inequality (\ref{penineconj}) holds for $(M,g)$.
\end{proposition}

\begin{proof}
Since the function $f(x)=x+Q^2x^{-1}$, $x> 0$, is non-decreasing in the interval $[|Q|,+\infty)$, the result follows from (\ref{afeuc}).
\end{proof}

We are thus left with the task of justifying (\ref{auxest}) under the conditions of Theorem \ref{main1}.
We first note that, as remarked by Gibbons \cite{Gi}, inequality (\ref{auxest}) holds true in dimension $n=3$ for {\rm any} Einstein-Maxwell initial data set, irrespective of it being embedded  in $\mathbb R^4$ or not; see also \cite{KWY2}. We now adapt his argument to our situation. First, the stability inequality applied to the function $\xi\equiv 1$ gives
\[
\int_{\Sigma_0}{\rm Ric}_g(\nu)d\Sigma\leq 0.
\]
On the other hand, in our situation the Gauss equation
reduces to
\[
R_h=R_g-2{\rm Ric}_g(\nu),
\]
so we obtain
\begin{equation}\label{sy}
\int_{\Sigma_0}R_gd\Sigma_0\leq \int_{\Sigma_0}R_kd\Sigma_0.
\end{equation}
We now use (\ref{decond}) and Cauchy-Schwarz
to get
\begin{eqnarray*}
\int_{\Sigma_0}R_gd\Sigma_0 &\geq & (n-1)(n-2)\int_{\Sigma_0}|E|_g^2d\Sigma_0\\
& \geq & (n-1)(n-2)|\Sigma_0|^{-1}\left(\int_{\Sigma_0}\langle E,\nu\rangle d\Sigma_0\right)^2\\
& = & (n-1)(n-2)|\Sigma_0|^{-1}\omega_{n-1}^2Q^2.
\end{eqnarray*}
If combined with (\ref{sy}), this means that
\begin{equation}\label{rew2}
Q^2\leq {\widetilde\Yc_{(\Sigma_0,k)}}
\mathfrak R_{\Sigma_0}^2,
\end{equation}
which is just a rewriting of (\ref{rew}).
In any case, this clearly shows that (\ref{condmain2}) implies (\ref{auxest}) and completes the proof of Theorem \ref{main1}.

In order to prove Theorem \ref{main2}, we simply observe that the Alexandrov-Fenchel inequalities in \cite{GL} give
\begin{equation}\label{sec}
\int_{\Sigma_0}R_kd\Sigma_0\geq d_n\left(\int_{\Sigma_0}Hd\Sigma_0\right)^{\frac{n-3}{n-2}},
\end{equation}
with the equality occurring if and only if $\Sigma_0$ is a round sphere, where
\[
d_n=\frac{(n-1)(n-2)\omega_{n-1}}{\left((n-1)\omega_{n-1}\right)^{\frac{n-3}{n-2}}}.
\]
This can be rewritten as
\[
\left(2c_n\int_{\Sigma_0}Hd\Sigma_0\right)^{-1}\geq \left(
\widetilde \Yc_{(\Sigma_0,k)}\right)^{-\frac{n-2}{n-3}}
\mathfrak R_{\Sigma_0}^{-1}.
\]
Thus, Theorem \ref{main2}
follows from (\ref{mainineq}) and (\ref{afeuc}).

\begin{remark}
\label{schwartz}
{\rm The proof of Theorem \ref{main0} also works if, instead of being star-shaped, we assume that the horizon $\Sigma_0\subset\mathbb R^n$ is outer minimizing as in Remark \ref{outer}. In  this case we must use the weak formulation of the inverse mean curvature flow \cite{HI} to carry out the argument. The details are explained in \cite{FS}, where the corresponding Alexandrov-Fenchel inequality extending (\ref{afeuc}) to this setting is established; see in particular the proof of their Lemma 8, which justifies the computation leading to (\ref{chain}).  We thus conclude that our main results actually hold in this generality, except for
Theorem \ref{main2}, since (\ref{sec}) is only known to follow from $2$-convexity in the star-shaped case \cite{GL}.}
\end{remark}

\end{document}